\documentclass[3p]{elsarticle}

\usepackage{amsfonts}
\usepackage{graphicx}
\usepackage{amsmath}
\usepackage{amssymb}
 \usepackage{amscd}
\newtheorem{theorem}{Theorem}[section]
\newtheorem{corollary}[theorem]{Corollary}

\newtheorem{lemma}[theorem]{Lemma}
\newtheorem{conjecture}[theorem]{conjecture}

\newtheorem{problem}[theorem]{Problem}
\newtheorem{proposition}[theorem]{Proposition}
\newtheorem{remark}[theorem]{Remark}

\newenvironment{proof}[1][Proof]{\textbf{#1.} }{\ \rule{0.5em}{0.5em}}

\begin{document}

\title{On the symmetric doubly stochastic matrices that are determined by their spectra}

\author[rvt]{Bassam Mourad\corref{cor1}\fnref{fn1}}
\ead{bmourad@ul.edu.lb}
\author[focal]{ Hassan Abbas}
\cortext[cor1]{Corresponding author}
 \fntext[fn1]{ Fax:+961 7 768174.}
\address[rvt]{Department of Mathematics, Faculty of Science V, Lebanese University, Nabatieh,
Lebanon}
\address[focal]{  Department of
Mathematics,  Faculty of Science I, Lebanese University, Beirut, Lebanon}

\begin{abstract}
A symmetric doubly stochastic matrix $A$ is said to be determined by its spectra if the only symmetric doubly stochastic matrices that are similar to $A$ are
of the form $P^TAP$ for some permutation matrix $P.$ The problem of characterizing such matrices is considered here. An ``almost'' the same but a more difficult problem was proposed by [ M. Fang, A note on the inverse eigenvalue problem for symmetric doubly stochastic matrices, Lin. Alg. Appl., 432 (2010) 2925-2927] as follows: ``Characterize all the $n$-tuples $\lambda=(1,\lambda_2,...,\lambda_n)$ such that up to a permutation similarity, there exists a unique symmetric doubly stochastic matrix with spectrum $\lambda.$'' In this short note, some general results concerning our two problems are first obtained. Then, we completely solve these two problems for the case $n=3.$ Some connections with spectral graph theory are then studied. Finally, concerning the general case, two open questions are posed and a conjecture is introduced.
  \end{abstract}

\begin{keyword} Doubly stochastic matrices \sep{ Inverse eigenvalue problem, Spectral characterization}
\MSC{15A12, 15A18, 15A51, 05C50}
\end{keyword}

\maketitle

\section{Introduction}

An $n\times n$ real matrix $A$ having each row and column sum equal to 1 is called {\it doubly quasi-stochastic.}
 If, in addition, $A$ is nonnegative then $A$
  is said to be {\it doubly stochastic. }  The set of all $n\times n$ doubly-stochastic matrices is denoted by $\Delta_n$ and the set of all the symmetric
 elements in $\Delta_n$ will be denoted by $\Delta^s_n.$ For $0\leq a \leq n,$ denote by $\Delta^s_n(a)$ to be the set of all elements of $\Delta^s_n$ with trace $a.$

 Three particular elements of $\Delta^s_n$ are of interest to us. The first is $I_n$ which is the $n\times n$ identity matrix and the second  $J_n$ which is the $n\times n$ matrix whose all entries are $\frac{1}{n}.$ The third is  $C_n$ which denotes the $n\times n$ matrix whose diagonal entries are all zeroes and whose off-diagonal entries are equal to $\frac{1}{n-1}$ (here for $n\geq 2$). In addition, let $e_{n}=\frac{1}{\sqrt{n}}(1,1,...,1)^T\in \mathbb{R}^n$ where $\mathbb{R}$ denotes the real line. For two matrices (and in particular for row vectors) $A$ and $B,$ the line-segment joining them is denoted by $[A,B],$ and for any two sets $E$ and $F$ we write $E-F$ to denote the set of elements in $E$ which are not in $F.$

Note that it is clear from the definition that an $n\times n$ real matrix $A$ is doubly quasi-stochastic if and only if $Ae_n=e_n$ and $e_n^TA=e_n^T$ if and only if $AJ_n=J_nA=J_n.$

  The {\it symmetric doubly stochastic inverse eigenvalue problem} asks which
sets of $n$ real numbers occur as the spectrum of an $n\times n$
symmetric doubly stochastic matrix. For more information on this problem see, e.g., \cite{fa,hw,jo,ka,mo,mou,mour,moura,mourad,mourada}.

  Regarding the inverse eigenvalue problem for symmetric doubly stochastic matrices, the following ``inaccurate'' proposition was presented in \cite{hw}.\\

 \begin{proposition} Let $\lambda=(1,\lambda_2,...,\lambda_n)$ be in $\mathbb{R}^n$ with $1> \lambda_2\geq ...\geq \lambda_n\geq -1.$ If
  $$\frac{1}{n}+\frac{1}{n(n-1)}\lambda_{2}+\frac{1}{(n-1)(n-2)}
\lambda_{3}+...+\frac{1}{(2)(1)}\lambda_{n} \geq0.
$$
then there is a positive (i.e. all of its entries are positive) doubly stochastic matrix $D$ in $ \Delta_n^s$ such that
$D$ has spectrum $\lambda$.
\end{proposition}
   In \cite{fa} the author presented a counterexample of the preceding proposition as follows.
\begin{theorem}
Let $\lambda=(1,0,-2/3).$ Then there does not exist a $3\times 3$ symmetric positive doubly stochastic matrix with spectrum $\lambda.$
 \end{theorem}
In~\cite{bru} it was pointed out that the 2-tuple $(1,-1)$ which is the spectrum of $C_2$ is also another counterexample.
 Moreover, it should be mentioned here that the proof of the preceding theorem (which is the main result of \cite{fa}) is done by showing that the matrix
  $A=\left ( \begin{array}{ccc} 0&2/3&1/3\\2/3&0&1/3\\1/3&1/3&1/3\\ \end{array} \right )$ has spectrum $\lambda$ and the only matrices in $\Delta_3^s$ that are similar to $A$ are of the form $P^TAP$ for some permutation matrix $P.$
Based on this,  the author suggested the following problem.
    \begin{problem} Characterize all the $n$-tuples $\lambda=(1,\lambda_2,...,\lambda_n)$ with $1\geq \lambda_2\geq ...\geq \lambda_n\geq -1$ such that up to a permutation similarity, there exists a unique symmetric doubly stochastic matrix $A$ with spectrum $\lambda$ $($such $\lambda$ is said to characterize $A$ permutationally or $A$ is said to be permutationally characterized by $\lambda$$).$
    \end{problem}
    Note now  that in the language of the preceding problem, the $3\times 3$  matrix $A$ presented above is permutationally characterized by $(1,0,-2/3).$

  Two matrices $A$ and  $B$ are said to be {\it permutationally similar} if $B=P^TAP$ for some permutation matrix $P.$ Next, we say that a  doubly stochastic matrix $A$ is \textit{determined by its spectra} (DS for short) in $ \Delta_n$  if for every element $B$ of $ \Delta_n$ which is similar to $A,$ then $B$ is permutationally similar to $A.$  If in addition $A$ is symmetric then $A$ is said to be DS in $ \Delta_n^s$ if $B\in \Delta_n^s$ is similar to $A$ implies that $B$ is permutationally similar to $A.$
     For a symmetric doubly stochastic matrix $A,$ obviously  $A$ is DS in $ \Delta_n$ implies that $A$ is DS in $  \Delta_n^s.$ However, it is not known whether the converse is true or false and though it is an interesting open problem, it will not be dealt with here. Also,
   though the problem of characterizing all doubly stochastic matrices  that are DS in $ \Delta_n$ is very interesting and we will touch on some aspects of this problem,   however here we are particularly more interested in the following problem which is very related to Problem 1.3.

  \begin{problem} Characterize all symmetric doubly stochastic matrices that are DS in $ \Delta_n^s.$
  \end{problem}

   All above problems appear to be very difficult and it seems that there is no systematic way under which these problems can be approached (see Section 3).  Practically nothing is known about them except perhaps what is mentioned earlier. In addition,
  we note that if a symmetric doubly stochastic matrix $A$ is permutationally characterized by $(1,\lambda_2,...,\lambda_n),$ then obviously $A$ is DS in   $ \Delta_n^s.$  So that in order to solve Problem 1.3, we need to solve Problem 1.4 first and then for every solution $X$ of this last problem, we have to find the spectrum of $X.$

  The rest of the paper is organized as follows. Section 2 is  mainly concerned with obtaining some general results for our two problems.  In Section 3, we completely solve Problem 1.3 and Problem 1.4  for the case $n=3$ which is one of the main results of this paper. In Section 4, we study the close connection of Problem 1.4   with spectral graph theory; more precisely with  ``regular graphs that are DS.'' We conclude in Section 5 by posing two open questions and by introducing a conjecture related to the general case.

\section{Some general results}

We start our study with the following two lemmas that explore some aspects of the spectral properties of doubly stochastic matrices and are consequences of the Perron-Frobenius theorem (see, e.g. \cite{mi}). But first recall that a square nonnegative matrix $A$ is {\it irreducible} if $A$ is not permutationally similar to a matrix of the form $\left ( \begin{array}{ccc} A_1&0\\A_2&A_3\\ \end{array} \right )$ where $A_1$ and $A_2$ are square. Otherwise, $A$ is said to be {\it reducible.}
\begin{lemma} Every doubly stochastic matrix is permutationally similar to a direct sum of irreducible doubly stochastic matrices.
\end{lemma}

\begin{lemma} Let $A$ be an $n\times n$ irreducible doubly stochastic matrix. If $A$ has exactly $k$ eigenvalues of unit modulus, then these are the $k$th roots of unity. In addition, if $k>1,$ then $k$ is a divisor of $n$ and $A$ is permutationally similar to a matrix of the form
$$  \left(
\begin{array}{ccccc}
0  & A_1 & 0  & \ldots    &  0 \\
0  & 0 & A_2  & \ldots    &  0 \\
\vdots & \vdots     & \vdots     & \ddots& \vdots \\
0  & 0 & 0  & \ldots    &  A_{k-1} \\
A_k  & 0 & 0  & \ldots    &  0 \\
\end{array}
\right)
$$
where $A_i$ is doubly stochastic of order $\frac{n}{k}\times\frac{n}{k}$ for $i=1,...k.$
\end{lemma}

As a result, we have the following.
\begin{theorem} Every  permutation matrix is DS in $ \Delta_n.$
\end{theorem}
\begin{proof}
Suppose first that  $A$ is an irreducible permutation matrix and let $X^{-1}AX$ be a doubly stochastic matrix, then clearly
  $X^{-1}AX$ is irreducible and all of its eigenvalues are of unit modulus. Therefore by the preceding lemma $X^{-1}AX$ is a permutation matrix. Now if $A$ is reducible then the proof can be completed by using Lemma 2.1.
\end{proof}

An immediate consequence is the following corollary.
\begin{corollary}
 Let $\lambda=(1,\lambda_2,...,\lambda_n)$ be in $\mathbb{R}^n$  where $\lambda_i\in \{-1,1\}$ for $i=2,...,n$ and such that $1+\lambda_2+...+\lambda_n\geq 0.$  Then $\lambda$ characterizes permutationally a vertex (i.e. symmetric permutation matrix) of $\Delta_n^s.$
\end{corollary}

\begin{lemma} Let $X$ be an invertible matrix such that $X^{-1}J_nX$ is symmetric doubly stochastic. Then $X^{-1}J_nX=J_n.$
\end{lemma}
\begin{proof} Since $X^{-1}J_nX$ is symmetric doubly stochastic, then by the spectral theorem for symmetric matrices, there exists an orthogonal matrix $U$ whose first column is $e_n$ and the remaining columns are orthogonal to $e_n$ (i.e. the sum of all components in each of the remaining columns is zero) such that $U^TX^{-1}J_nXU=(1\oplus 0_{n-1})$ where $0_{n-1}$ is the $n-1\times n-1$ zero matrix. Hence $X^{-1}J_nX=U(1\oplus 0_{n-1})U^T.$ But then a simple check shows that $U(1\oplus 0_{n-1})U^T=J_n$ and the proof is complete.
\end{proof}

\begin{corollary} The matrices $I_n,$ $J_n$ and $C_n$  are DS in $ \Delta_n^s.$
\end{corollary}
\begin{proof} The first part is obvious, and the second part follows from the preceding lemma.  For the third part, we note that $C_n=\frac{n}{n-1}J_n-\frac{1}{n-1}I_n$ and then for any invertible matrix $X$ such that $X^{-1}C_nX$ is symmetric doubly stochastic we obtain
$X^{-1}C_nX=\frac{n}{n-1}X^{-1}J_nX-\frac{1}{n-1}I_n.$ Therefore $X^{-1}C_nX+\frac{1}{n-1}I_n=\frac{n}{n-1}X^{-1}J_nX$ and where the left-hand side is a nonnegative matrix with row and column sum equals to $1+\frac{1}{n-1}.$ Thus $X^{-1}J_nX=\frac{n-1}{n}(X^{-1}C_nX+\frac{1}{n-1}I_n)$ is symmetric doubly stochastic and then by the preceding lemma, the proof is complete.
\end{proof}

 Next we need the following auxiliary materials.
 \begin{lemma} The inverse of an invertible doubly quasi-stochastic matrix is doubly quasi-stochastic.
 \end{lemma}
 \begin{proof}
 Multiplying to the left of  $AA^{-1}=I_n$ by $J_n$ we obtain $J_nAA^{-1}=J_n.$ Since $A$ is doubly quasi-stochastic, then $J_nA^{-1}=J_n.$  Similarly, multiplying to the right of $A^{-1}A=I_n$ by $J_n,$ we obtain $A^{-1}J_n=J_n.$ Thus $A^{-1}$ is doubly quasi-stochastic.
 \end{proof}
 \begin{lemma} If $A$ is an $n\times n$ irreducible doubly stochastic matrix such that $B=X^{-1}AX$ is doubly stochastic for some invertible matrix $X$, then there exists a doubly stochastic matrix  $Y$ such that $B=Y^{-1}AY.$
 \end{lemma}
 \begin{proof} See [12, Theorem 4.1, p. 123].
 \end{proof}
 \begin{corollary} The matrices $I_n,$ $J_n$ and $C_n$  are DS in $ \Delta_n.$
\end{corollary}
\begin{proof} If $B=X^{-1}J_nX$ is doubly stochastic, then by the preceding lemma, there exists $Y\in \Delta_n$ such that $B=Y^{-1}J_nY.$ Hence $B=J_n.$ The rest of proof can be completed by using a similar argument as that of Corollary 2.6.
\end{proof}

 It is easy to see that each symmetric doubly stochastic matrix $D_a$ of trace $a$ which lies on the line-segment joining $I_n$ to $C_n$ has the property that $0\leq a\leq n.$ Also recall that $J_n=\frac{n-1}{n}C_n+\frac{1}{n}I_n$ so that  $J_n$ is on this line-segment $[I_n,C_n].$  With this in mind, we have the following theorem.
 \begin{theorem} Any point $D_a$ that lies on the line-segment  $[I_n,C_n]$  is DS in $ \Delta_n$ and hence it is also DS in $ \Delta_n^s.$
  \end{theorem}
\begin{proof} We split the proof into two cases.
\begin{itemize}  \item For $0\leq a\leq 1,$ then $D_a$ is a convex combination of $J_n$ and $C_n.$ From the trace of $D_a$, we easily obtain $D_a=aJ_n+(1-a)C_n$ so that
$D_a=aJ_n+(1-a)\frac{n}{n-1}J_n-\frac{1}{n-1}I_n$ or $D_a=\frac{n-a}{n-1}J_n-\frac{1-a}{n-1}I_n.$ Note that $D_a$ is a positive matrix and so it is irreducible. Now if $X$ is an invertible matrix such that $B=X^{-1}D_aX$ is doubly stochastic, then by the preceding lemma, there exists  a doubly stochastic matrix $Y$ such that $B=Y^{-1}D_aY.$ Hence $B=X^{-1}D_aX=Y^{-1}D_aY=\frac{n-a}{n-1}Y^{-1}J_nY-\frac{1-a}{n-1}Y^{-1}I_nY.$ Thus $X^{-1}D_aX=D_a$ and this shows that $D_a$ is DS in $ \Delta_n.$
 \item For $1\leq a\leq n,$ then $D_a$ is a convex combination of $I_n$ and $J_n.$ From the trace of $D_a$, it is easy to see that in this case $D_a=\frac{a-1}{n-1}I_n+(1-\frac{a-1}{n-1})J_n,$  and that the proof can be completed in a similar way to that of the previous case.
     \end{itemize}
\end{proof}

Knowing that the eigenvalues of $C_n$ are given by $(1,-\frac{1}{n-1},...,-\frac{1}{n-1}),$ then we have the following conclusion.
\begin{corollary} Let $\lambda$ be any point that lies on the line-segment $[(1,...,1),(1,-\frac{1}{n-1},...,-\frac{1}{n-1})]$ of $\mathbb{R}^n.$  Then $\lambda$ characterizes a unique element of $\Delta_n^s.$
\end{corollary}

It should be noted that if two doubly stochastic matrices $A$ and $B$ are DS in $\Delta_n^s$ (or $\Delta_n$) then   their {\it direct sum} $A\oplus B$  may not be DS in $\Delta_n^s.$ To see this, it suffices to check that in $\Delta_{2k}^s,$ the matrices $J_2\oplus J_{2k-2}$ and $J_k\oplus J_k$  have the same spectrum so that they are similar (as they are symmetric). Moreover, for $k\geq 3,$ $\frac{1}{k}$ is an entry of the latter and is not an entry of the first so that they not permutationally similar. However, we have the following.
\begin{theorem} The matrix $C_n\oplus C_n$ is DS in $\Delta_{2n}.$
\end{theorem}
\begin{proof} If  $Z\in \Delta_{2n}$  is similar to $C_n\oplus C_n$  then obviously the spectrum of $Z$ is $(1,1,-1/(n-1),...,-1/(n-1)).$ Since $Z$ is reducible and has the eigenvalue  1 repeated twice, then $Z$ is permutationally similar to a direct sum of two doubly stochastic matrices  $A$ and $ B.$ But the traces of $A$ and $B$ are zeroes so that necessarily the spectrum of $A$ and $B$ is the same and is equal to $(1,-1/(n-1),...,-1/(n-1)).$ By corollary 2.9, $A=B=C_n.$
\end{proof}

Using a virtually identical proof to that of the preceding theorem, we conclude by mathematical induction  the following.
\begin{corollary} For any positive integers $n_1, ..., n_k,$ the matrix  $C_{n_1}\oplus ...\oplus C_{n_k}$ is DS in $\Delta_{n_1+...+n_k}.$
\end{corollary}

Our next result is concerned with some other doubly stochastic matrices that are DS in  $\Delta_{2n}.$  For this purpose, we introduce the following notations. In $\Delta_{2n}^s,$ define $I=\left ( \begin{array}{cc} 0&I_n\\I_n&0\\ \end{array} \right ),$  $J=\left ( \begin{array}{cc} 0&J_n\\J_n&0\\ \end{array} \right )$ and $C=\left ( \begin{array}{cc} 0&C_n\\C_n&0\\ \end{array} \right )$ then it can be easily checked that $C=\frac{n}{n-1}J-\frac{1}{n-1}I$ so that $J$ belongs to the line-segment $[I,C].$
With this in mind, we conclude with the following result.
\begin{theorem}  Any point on the line-segment $[I,C]$ is DS in $\Delta_{2n}$  and hence in $\Delta_{2n}^s.$
\end{theorem}
\begin{proof} We first prove that $J$ and  $I$ are DS in $\Delta_{2n}.$ For, if there exists $Z\in \Delta_{2n}$ which is similar to $J,$ then obviously the spectrum of $Z$ is $(1,0,...,0,-1)\in\mathbb{R}^{2n} .$ By Lemma 2.2, $Z$ is permutationally similar to a matrix of the form $\left ( \begin{array}{cc} 0&D\\D^T&0\\ \end{array} \right )$ where $D\in \Delta_n. $ So that $DD^T\in \Delta_n^s$  and its eigenvalues are $(1,0,...,0)$ and therefore $DD^T$ is similar to $J_n.$   By Lemma 2.5, $DD^T=J_n$ and since rank($D$)=rank($DD^T$)=rank($J_n$)=1, then $D=J_n$ and therefore $Z=J.$

Now suppose that $S\in \Delta_{2n}$ is similar to $I,$ then $S$ has spectrum
 $(\underbrace{1,...,1}_{\text{n times}},\underbrace{-1,...,-1}_{\text{n times}})$ and therefore $S$ is permutationally similar to                    $\underbrace{C_2\oplus ...\oplus C_2}_{\text{n times}}$ but this in turn means that $I$ and $S$ are permutationally similar. Since $C=\frac{n}{n-1}J-\frac{1}{n-1}I$ then a similar proof to that of Corollary 2.6 shows that $C$ is also DS in $\Delta_{2n}.$ Finally, using a similar argument to that of Theorem 2.10, the proof can be easily completed.
\end{proof}

Recall that two matrices are {\it cospectral} if they have the same spectra.
We conclude this section by proving that  a symmetric doubly stochastic matrix that is DS in  $\Delta_n$ may be cospectral to another element of  $\Delta_n.$ But first, we need the following result for which the proof can be found in \cite{zh}.
\begin{lemma} Let $M=\left ( \begin{array}{cc} A&B\\C&D\\ \end{array} \right )$ where $A$ and $D$ are square. If $AC=CA,$
 then  $$\det(M)=\det(AD-CB).$$
\end{lemma}
\begin{theorem} Let $A$ be in $\Delta_n$ and define $M=\left ( \begin{array}{cc} 0&J_n\\A&0\\ \end{array} \right )$ and let $J=\left ( \begin{array}{cc} 0&J_n\\J_n&0\\ \end{array} \right )$ be  as defined earlier. Then $M$ and $J$ are cospectral.
\end{theorem}
\begin{proof} The characteristic polynomial of $M$ is given by $p_M(\lambda)=\det\left ( \begin{array}{cc} -\lambda I_n&J_n\\A&-\lambda I_n\\ \end{array} \right ).$ By the preceding lemma, $p_M(\lambda)=\det(\lambda^2I_n-AJ_n)=\det(\lambda^2I_n-J_n)$ i.e. $\lambda^2$ is an eigenvalue of $J_n.$ On the other hand,  $p_J(\lambda)=\det(\lambda^2I_n-J_nJ_n)=\det(\lambda^2I_n-J_n).$ Thus $M$ and $J$ are cospectral.
\end{proof}

\section{ Particular cases}

Recall  that Birkhoff's theorem states that $\Delta_n$ is a convex polytope of dimension $(n-1)^2$ where
its vertices are the $n\times n$ permutation matrices.
On the other hand, $\Delta_n^s$ is a convex polytope of dimension $\frac{1}{2}n(n-1)$, and its
vertices were determined in ~\cite{kat,cru} where it
is proved that if $A$ is a vertex of $\Delta_n^s$, then
 $A= \frac{1}{2}(P+P^T)$ for some permutation matrix $P$, although not every
 $\frac{1}{2}(P+P^T)$ is a vertex.

\subsection{ The case $n=2$}
 It is easy to see $\Delta_2=\Delta_2^s$ i.e. every $2\times 2$ doubly stochastic matrix is necessarily symmetric. Moreover, $\Delta_2$ is the line-segment joining $I_2$ to $C_2.$ So that every $2\times 2$ doubly stochastic matrix is determined by its spectra and every point of the line-segment joining $[(1,1),(1,-1)]$ characterizes a unique $2\times 2$  doubly stochastic matrix.

\subsection{ The case $n=3$}
Here we solve completely Problem 1.3 and hence Problem 1.4 for the case $n=3.$

 The convex polytope  $\Delta^s_3$ sits in
 the  6-dimensional vector space of all $3\times 3$ real symmetric matrices, and
following ~\cite{kat, cru} $\Delta^s_3$  is the convex hull
of the following matrices:  \\
$$I_3, \mbox{ }
X=\left ( \begin{array}{ccc} 1&0&0\\0&0&1\\0&1&0\\ \end{array} \right ), \mbox{ }
 Y=\left ( \begin{array}{ccc} 0&0&1\\0&1&0\\1&0&0\\ \end{array} \right ), \mbox{ }
Z=\left ( \begin{array}{ccc} 0&1&0\\1&0&0\\0&0&1\\ \end{array} \right ), \mbox{ }
 C_3=\left ( \begin{array}{ccc} 0&1/2&1/2\\1/2&0&1/2\\1/2&1/2&0\\
 \end{array} \right ).$$

Our main observation is the following:
\begin{lemma} $J_3=\frac{1}{3}(X+Y+Z)$, $C_3=\frac{3}{2}J_3-\frac{1}{2}I_3$
 and the triangle $ XYZ$ is equilateral with respect to the Frobenius norm.
 In addition, $I_3C_3$ is an axis of symmetry for
$\Delta_3^s$, and every point on $I_3C_3$ commutes with all other points
 in $\Delta_3^s$.
 \end{lemma}
Thus it is clear from the preceding lemma that $\Delta_3^s$ is 3-dimensional and has the shape seen
 in Figure \ref{fig: my fig1}.

\begin{figure}
\begin{center}
\includegraphics[angle=270,scale=0.4]{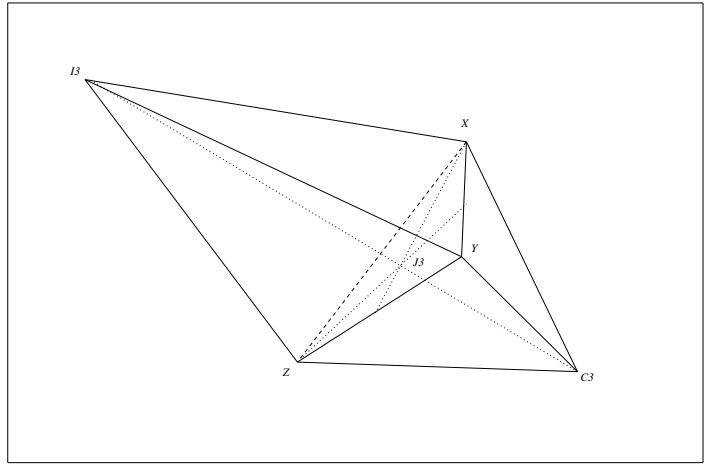}
\caption{{\footnotesize The shape of $\Delta_3^s $.}}
\label{fig: my fig1}
\end{center}
\end{figure}

Our next goal is to prove that the only symmetric doubly stochastic matrices of $\Delta_3^s(1)$ (which is the closed triangle $XYZ$) that are DS in $\Delta_3^s$  are $J_3$ and its vertices  $X,$  $Y$ and $Z.$

Using Maple for example, it is easy to check the following lemma.
\begin{lemma} For $0\leq x \leq 1$ and  $0\leq y\leq 1$ with $0\leq x+y\leq 1,$ the symmetric doubly stochastic matrix $A=xX+yY+(1-x-y)Z$ has eigenvalues
$$\left ( 1,\sqrt{3x^2+3y^2+3xy+1-3x-3y},-\sqrt{3x^2+3y^2+3xy+1-3x-3y}\right ).$$
\end{lemma}
Now if we let the {\it domain} $D$ be defined by $0\leq x \leq 1,$   $0\leq y\leq 1$ and $0\leq x+y\leq 1,$ then it is easy to see that $D$ is the closed triangle whose vertices are $O=(0,0), A=(1,0)$ and $B=(0,1).$   Define the function $f$ over $D$ by: $$f(x,y)=3x^2+3y^2+3xy+1-3x-3y.$$
 Concerning the function $f,$ we have the following.
\begin{lemma} Over the domain $D,$ the function $f$ has zero as  absolute minimum and 1 as an absolute maximum. Thus over the domain $D,$ we have
 $0\leq 3x^2+3y^2+3xy+1-3x-3y\leq 1.$
\end{lemma}
\begin{proof}
Since $f$ is differentiable then the only places where $f$ can assume these values are points inside $D$ where the first partial derivatives satisfy  $f_x=f_y=0,$ and points on the boundary.
\begin{itemize}
\item{Potential points inside $D$:} Solving the system
$$
\left \{
\begin{array}
[c]{l}
6x+3y-3=0 \\
3x+6y-3=0\\
\end{array}
\right .
$$
yields  the unique solution $x=y=\frac{1}{3}$ with $f(\frac{1}{3},\frac{1}{3})=0.$
\item{Potential points on the boundary of $D$:} We have to check the 3 sides of the triangle $OAB$ one side at a time.\\
1. On the segment $[O,A],$ $fx,y)=f(x,0)=3x^2-3x+1$ which can be regarded as a function of $x$ where $0\leq x\leq 1,$ and such that its derivative $f'(x,0)=6x-3=0$ for $x=1/2.$ Therefore we have 3 potential points where their images by $f$ are given by $f(0,0)=1,$ $f(1,0)=1,$ and $f(1/2,0)=1/4.$\\
2. On the segment $[O,B],$ clearly (as $x$ and $y$ play a symmetric role in the function $f(x,y)$) we obtain the following potential points:  $f(0,0)=1,$ $f(0,1)=1,$ and $f(0,1/2)=1/4.$ \\
3. On the segment $[A,B],$ we have already accounted for the values of $f$ at the endpoints of $[A,B],$ so that we only need to look at the interior points of $[A,B].$  Clearly $f(x,1-x)=3x^2-3x+1$ and $f'(x,1-x)=6x-3=0$ for $x=1/2.$ Hence, $(1/2,1/2)$ is the final potential point with $f(1/2,1/2)=1/4.$
\end{itemize}
Thus our claim is valid.
\end{proof}
\begin{remark} The surface $z=f(x,y)$ where $(x,y)\in D$ and any horizontal plane $z=d$ where $0\leq d\leq 1$ intersect
at exactly one point which is $(1/3,1/3)$ for $d=0$ and intersect at the three points $(0,0),$ $(1,0)$ and $(0,1)$ for $d=1.$  Moreover, for $0< d < 1$ they intersect in an infinite number of points (see Figure \ref{fig: my fig2}).
\end{remark}
\begin{figure}
\begin{center}
\includegraphics[scale=0.4]{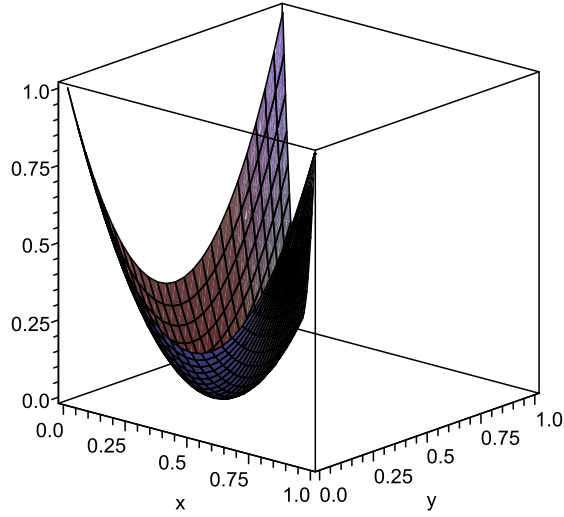}
\caption{{\footnotesize The surface $z=f(x,y)$ over $D.$}}
\label{fig: my fig2}
\end{center}
\end{figure}

As a consequence, we have the following corollary.

\begin{lemma} The only elements of $\Delta_3^s(1)$  that are DS  in  $\Delta_3^s$ are $J_3$ and   $X,$  $Y$ and $Z.$
\end{lemma}
\begin{proof} First note that the line segments $[J_3,X],$ $[J_3,Y]$ and $[J_3,Z]$ are permutationally similar as $X,$ $Y$ and $Z$ are and any point outside these line segments can not be permutationally similar to a point on them (from the geometry of the triangle $XYZ$). Let $M$ be any point in $\Delta_3^s(1)$ and consider the following two cases:
\begin{itemize}
 \item If  $ M\in \Delta_3^s(1)- [J_3,X]\cup [J_3,Y]\cup [J_3,Z]$ then  $M=xX+yY+(1-x-y)Z$  for some $(x,y)\in D.$  Define  $\alpha= \sqrt{f(x,y)}$ where $(x,y)$ varies over the domain $D.$ Then by the preceding lemma, $0\leq \alpha \leq 1.$  Moreover, it is easy to check that the matrix $N$ given by $N=\alpha X+(1-\alpha)J_n$ has eigenvalues $(1,\alpha,-\alpha).$ So that by Lemma 3.2 the two symmetric doubly stochastic matrices matrices $M$ and $N$
  have the same spectrum. Since we are dealing with symmetric matrices, then they are similar. Thus in this case $M$ is not DS in  $\Delta_3^s.$\\
 \item For the case where $M$ is in $ [J_3,X]\cup [J_3,Y]\cup [J_3,Z]-\{J_3, X, Y, Z\},$ without loss of generality let $M=d X+(1-d)J_3$ for some $0<d <1.$  We want to show that there there exists $K\in \Delta_3^s(1) $ which is similar to $M$ but not permutationally similar. For, let $K=x X+yY+(1-x-y)Z$ where $(x,y)$ varies over $D.$ First the condition on $(x,y)\in D$ in terms of $d$ for which $M$ and $K$ are similar is given by $d=\sqrt{f(x,y)}.$ Such $(x,y)$ always exists due to the continuity of $f(x,y)$ in $D.$  Also we want to impose the other constraint that  at least one entry of $M$ is not an entry of $K$ or vice versa so that they are not permutationally similar. Clearly $M=\left ( \begin{array}{ccc} 1/3+2/3d&1/3-1/3d&1/3-1/3d\\1/3-1/3d&1/3-1/3d&1/3+2/3d\\1/3-1/3d&1/3+2/3d&1/3-1/3d\\ \end{array} \right )$ and
   $K=\left ( \begin{array}{ccc} x&1-x-y&y\\1-x-y&y&x\\y&x&1-x-y\\ \end{array} \right ),$
   and since $M$ has at most two distinct entries which are $1/3+2/3d$ and $1/3-1/3d$ so that we need to impose the constraint that  $x\neq 1/3+2/3d$ and $x\neq 1/3-1/3d.$  An inspection shows that $x=1/3+2/3d$ or $x=1/3-1/3d$ if and only if  $(3x-1)^2=4f(x,y)$ or $(3x-1)^2=f(x,y)$ if and only if $(x-(2y-1))^2=0$ or $(x-y)(2x+y-1)=0.$ So that our second constraint amounts to $x$ not being an element of $\{y, 2y-1, (1-y)/2\}.$ Thus we only need to exclude these 3 particular values of $x$ and since $0<d <1$ then by Remark 3.4, an infinite number of such $x$ exists (since each of the 3 planes $x=y,$ $x=2y-1$ and $x=(1-y)/2$ intersects the curve $d=\sqrt{f(x,y)}$ in a finite number of points) so that we conclude that $M$ is not DS in $\Delta_3^s.$
   \end{itemize}
   Finally, $X,$ $Y$ and $Z$ are DS by Theorem 2.3 and $J_3$ is DS by Corollary 2.9.
   \end{proof}

For $1\leq a\leq 3,$ let $\Delta_3^s(a)$ intersect $[I_3,J_3],$ $[I_3,X],$ $[I_3,Y]$ and $[I_3,Z]$ in $D_a,$  $X_a,$ $Y_a$ and $Z_a$ respectively (for $0\leq a\leq 1,$  we only need to replace $I_3$ by $C_3,$ in this statement).
 Then clearly $\Delta_3^s(a)$ is the closed triangle $X_aY_aZ_a$ and the 3 vertices  $X_a,$ $Y_a,$ and $Z_a$  are permutationally similar since $X,$ $Y,$ and $Z$ are.
With these notations, we have the following.
\begin{lemma} The only points of $\Delta_3^s(a)$ that are DS in $\Delta_3^s$ are $\{ D_a, X_a, Y_a, Z_a\}.$
\end{lemma}
\begin{proof} For $1\leq a\leq 3,$ let $M_a$ be any point in $\Delta_3^s(a)-\{ D_a, X_a, Y_a, Z_a\},$ and let the line through $I_3$ (resp. $C_3$ for $0\leq a\leq 1$) and $M_a$  intersect $\Delta_3^s(1)$ in $M.$ Clearly $M$ is in $\Delta_3^s(1)-\{ J_3, X, Y, Z\}$ and $M$ is not DS by the preceding lemma. Then there exists $N$ in $\Delta_3^s(1)-\{ J_3, X, Y, Z\}$ such that $N$ and $M$ are similar but not permutationally similar. Let $N_a$ be the intersection of $[I_3,N]$ (resp. $[C_3,N]$ for $0\leq a\leq 1$) with $\Delta_3^s(a),$ then clearly $M_a$ and $N_a$ are similar but not permutationally similar.

If $M_a= D_a,$ then $D_a$ is DS in $\Delta_3^s.$  Now if  $M_a\in \{ X_a, Y_a, Z_a\},$ then it is enough to study the case where $M_a= X_a.$ If there exists
$N_a\in \Delta_3^s(a)-\{ D_a, X_a, Y_a, Z_a\},$ such that $X_a$ and $N_a$ are similar but not permutationally similar, then there exists $N\in \Delta_3^s(1)-\{ J_3, X, Y, Z\}$ such that $X$ and $N$ are similar but not permutationally similar which is a contradiction to $X$ being DS in $\Delta_3^s.$
\end{proof}

From the preceding 3 lemmas, we conclude one of our main results which completely solves Problem 1.4 in the case $n=3.$
\begin{theorem} The only symmetric doubly stochastic matrices that are DS in $\Delta_3^s$  are those lying on one of the following line-segments $[I_3,X],$ $[I_3,Y],$ $[I_3,Z],$ $[C_3,X],$ $[C_3,Y],$ $[C_3,Z],$  or $[I_3,C_3].$
\end{theorem}
As a conclusion, we solve Problem 1.3 for the case $n=3.$
\begin{corollary} The only points of $\mathbb{R}^3$  that  characterize permutationally  elements of $\Delta_3^s$ are those belonging to $[(1,1,1),(1,1,-1)]\cup[(1,-1/2,-1/2),(1,1,-1)]\cup[(1,1,1),(1,-1/2,-1/2)].$
\end{corollary}
\begin{proof} It is enough to check that the spectrum of any of the 3 line-segments $[I_3,X],$ $[I_3,Y],$ $[I_3,Z]$ is $[(1,1,1),(1,1,-1)],$ and the spectrum of any of $[C_3,X],$ $[C_3,Y],$ $[C_3,Z]$ is $[(1,-1/2,-1/2),(1,1,-1)].$ The last part is true by Corollary 2.11.
\end{proof}

\section{ Connections with spectral graph theory}

In this section, we present some close connections between Problem 1.4 and the topic known  ``regular graphs that are DS'' (see, e.g.,~\cite{br,cv,va,van}).
  First let us introduce some related notations (see, e.g. \cite{ba}). The adjacency matrix of a simple
graph $G$ will be denoted by $A(G)$ which is a symmetric nonnegative (0,1)-matrix and its eigenvalues
$\lambda_1,...\lambda_n$ form the spectrum of $G$ which is a {\em multiset} and will be
denoted by $\sigma(G).$ The graph $G$ is called {\it integral} if all of its eigenvalues are integers, and it is called {\it circulant} if $A(G)$ is circulant. In addition, $G$ is said to be $k$-{\it regular} if the degree of each of its vertices is $k.$  A  {\it  strongly regular} graph $G$ with parameters $(v,k,\lambda,\mu)$ is a $k$-regular graph which is not complete nor edgeless and satisfying the following two conditions:\\
(i) For each pair of adjacent vertices there exist $\lambda$ vertices adjacent to both.\\
(ii) For each pair of non-adjacent vertices there exist $\mu$ vertices adjacent to both.

We use the usual notation $K_n$ to denote the {\it complete} graph on $n$ vertices where each vertex is connected to all other vertices. Moreover, the {\it complete bipartite} graph $K_{n_1,n_2}$ has vertices partitioned into two
subsets $V_1$ and $V_2$ of $n_1, n_2$ elements each, and two vertices are
adjacent if and if only if one is in $V_1$ and the other is in $V_2.$

 Two graphs are said to be {\it isomorphic} if and only if their adjacency matrices are permutationally similar. Two graphs  are said to be {\it cospectral or isospectral} if they have the same spectrum. A graph $G$ is said to be DS if any graph $H$ which is cospectral to $G$ is isomorphic to $G.$ In general, the problem of determining whether a graph $G$ is DS or not is still open though many partial results are known (see \cite{van} for the latest developments on this problem).

  We are particularly interested in the subproblem of finding which regular graphs are DS due to its link with Problem 1.4. To explain this, we need some more notations. But first recall that if $G$ is  a $k$-regular graph with $n$ vertices, then the spectral radius of $A(G)$ equals $k$ and it is an eigenvalue of $G$ with corresponding unit eigenvector equals to $e_{n}.$
    Now let $\Omega_n^s(k)$ be the set of all $n\times n$ nonnegative symmetric matrices with each row and each column equals to $k,$  and   $\Lambda_n(k)$ denote the subset of $\Omega_n^s(k)$ formed by of all (0,1)-matrices with $k$ 1's in each row and each column. In addition, let  $\Lambda_n^0(k)$ be the set of those elements of $\Lambda_n(k)$ that have  zero trace. Then clearly $A\in \Lambda_n^0(k)$ if and only if $A$ is the adjacency matrix of some $k$-regular graph with $n$ vertices and $k$ edges.  Also note that if $A$ is in $\Omega_n^s(k)$ if and only if $\frac{1}{k}A$ is an element of $\Delta_n^s.$ So that if we extend the notion of DS to all elements of $\Omega_n^s(k),$  then obviously $A\in \Lambda_n(k)$ is DS in $\Omega_n^s(k)$ if and only if  $\frac{1}{k}A$ is DS in $\Delta_n^s.$ Also, note that a $k$-regular graph $G$ is DS if and only if $A(G)$ is DS in $\Lambda_n^0(k)$

 It is well-known that 1-regular graphs are DS (see \cite{cv}); a fact that can be easily  derived from Theorem 2.3. In addition, the fact that the complete graph $K_n$ is DS can be seen from Corollary 2.9 since $\frac{1}{k}A(K_n)=C_n.$  On the one hand, a disjoint union of complete graphs is DS; a fact that can be deduced from Corollary 2.13, and on the other hand, $K_{n,n}$ is DS by Theorem 2.14.

  Although proving that graphs are DS is a much more harder task than just showing they are not DS and the same is true for Problem 1.4, one can benefit from the fact that cospectral regular graphs that are not isomorphic (i.e. cospectral mates) give rise to symmetric doubly stochastic matrices that are not DS in $\Delta_n^s.$ So that all known results concerning finding cospectral mates for regular graphs can lead to exclude elements from $\Delta_n^s(0)$ as solutions to Problem 1.4. In what follows,  we mention among the many such situations, 3 particular examples (see \cite{br} for other situations). The first is concerned with strongly regular graphs  where
 it is well known that connected strongly regular graphs with parameters $(v,k,\lambda,\mu)$ have eigenvalue k appearing once and two other eigenvalues with prescribed multiplicity. In general there are many non-isomorphic graphs for a fixed parameter and the number of non-isomorphic graphs can grow dramatically (see e.g. \cite{br}).
  The second deals with cospectral integral regular graphs where for example in \cite{wa}(see also the references within) the authors prove the existence of infinitely many pairs of cospectral integral graphs which results in the existence of infinitely many pairs of symmetric doubly stochastic matrices that are not DS in $\Delta_n^s.$ The final case is concerned with circulant graphs (which are regular)  where in \cite{ju} it is proved that there are infinitely many cospectral non-isomorphic circulant graphs.

\section{Two related open questions and a conjecture}

We conclude this paper with the following two open questions  for which the answer to any of them can help shed some light on Problem 1.4 for general $n.$\\
(1) If $G$ is a $k$-regular graph that is DS. Does this imply that $\frac{1}{k}A(G)$ is DS in $\Delta_n^s$?\\
 Note that as mentioned earlier this is true for 1-regular graphs, disjoint union of complete graphs, the graphs $K_n$  and $K_{n,n}.$\\
(2) What are the elements of $\Delta_n^s(0)$  that are DS in $\Delta_n^s$? \\
we know that $C_n$ and the zero trace $n\times n$ permutation matrices are among these ones.

Finally,  based on the solution for the case $n=3,$ we propose the following conjecture.
\begin{conjecture} For $0< a \leq n,$ the only elements of $\Delta_n^s(a)$ that are DS in $\Delta_n^s$ are points on the line segments $[I_n,C_n],$ $[I_n,P]$ and $[C_n,P]$ where $P$ is a vertex of $\Delta_n^s.$
\end{conjecture}

\section*{Acknowledgments}
This work is supported by the Lebanese University research grants program for the Discrete Mathematics and Algebra research group.

\end{document}